\documentclass[Svgc_modified]{Svgc_modified}
\usepackage{amsmath}
\usepackage{graphicx}
\usepackage{enumitem}
\usepackage{subfigure}
\usepackage{stmaryrd}
\def\1{$\ }
\def\2{\ $}

\def\N{\mbox{\makebox[.2em][l]{I}N}}

\def\cK{{\cal K}}

\def\cC{{\cal C}}
\def\cE{{\cal E}}
\def\cH{{\cal H}}

\def\cP{{\cal P}}

\def\hqed{\hfill\framebox(6,6){\ }}

\def\ul{\underline}

\def\inN{N_D^-}
\def\outN{N_D^+}

\def\eie{e^{\mbox{\tiny{\it EI}}}}
\def\eiE{\cE^{EI}}

\def\kEI{k^{\mbox{\tiny{\it EI}}}}
\def\ei{{EI}}

\begin{document}
\title{Edge intersection hypergraphs -- a new hypergraph concept}
\author{Martin Sonntag\inst{1} \and Hanns-Martin Teichert\inst{2}}


%
%
\institute{Faculty of Mathematics and Computer Science, Technische Universit\"{a}t Bergakademie Freiberg, Pr\"{u}ferstra\ss e 1, 09596 Freiberg, Germany
\and Institute of Mathematics, University of L\"{u}beck, Ratzeburger Allee 160, 23562 L\"{u}beck, Germany
}
\maketitle
\begin{abstract}
If $\cH=(V,\cE)$ is a hypergraph, its {\it edge intersection hypergraph} $EI(\cH)=(V,\eiE)$ has the edge set $\eiE=\{e_1 \cap e_2 \ |\ e_1, e_2 \in \cE \ \wedge \ e_1 \neq e_2  \ \wedge  \ |e_1 \cap e_2 |\geq2\}$. Besides investigating several  structural properties of edge intersection hypergraphs,  we prove that all trees but seven exceptional ones are edge intersection hypergraphs of  3-uniform hypergraphs.
\end{abstract}
%
\section{Introduction and basic definitions}
All hypergraphs $\cH = (V(\cH), \cE(\cH))$ and (undirected) graphs $G=(V(G), E(G))$  considered in the following may have isolated vertices but no multiple edges or loops.

A hypergraph $\cH = (V, \cE)$ is {\em $k$-uniform} if all hyperedges $e \in \cE$ have the cardinality $k$.
Trivially, any 2-uniform hypergraph $\cH$ is a graph.
The {\em degree} $d(v)$ (or $d_{\cH}(v)$) of a vertex $v \in V$ is the number of hyperedges $e \in \cE$ being incident to the vertex $v$.
$\cH$ is {\em linear} if any two distinct hyperedges $e, e' \in \cE$ have at most one vertex in common.

If $\cH=(V,\cE)$ is a hypergraph, its {\it edge intersection hypergraph} $EI(\cH)=(V,\eiE)$ has the edge set $\eiE=\{e_1 \cap e_2 \ |\ e_1, e_2 \in \cE \ \wedge \ e_1 \neq e_2  \ \wedge \ |e_1 \cap e_2 |\geq2\}$. For $k \ge 1$, the $k$-th iteration of the $EI$-operator is defined to be $EI^k(\cH):= EI(EI^{k-1}(\cH))$, where  $EI^0(\cH):= \cH$. Moreover, the {\em EI-number} $\kEI(\cH)$ is the smallest $k \in \N$ such that $\cE(EI^k(\cH)) = \emptyset.$

Let $e = \{ v_1, v_2, \ldots, v_l \} \in \eiE$ be a hyperedge in $EI(\cH)$. By definition, in $\cH$ there exist (at least) two hyperedges $e_1, e_2 \in \cE(\cH)$ both containing all the vertices $v_1, v_2, \ldots, v_l$, more precisely $\{ v_1, v_2, \ldots, v_l \} = e_1 \cap e_2 $. In this sense, the hyperedges of $EI(\cH)$ describe sets $\{ v_1, v_2, \ldots, v_l \}$ of vertices having a certain, "strong" neighborhood relation in the original hypergraph $\cH$.

As an application, we consider a hypergraph $\cH = ( V, \cE)$ representing a communication system. The vertices $v_1, v_2, \ldots, v_n \in V$ and the hyperedges $e_1, e_2, \ldots, e_m \in \cE$ correspond to $n$ people and to $m$ (independent) communication channels, respectively. A group $\{ v_{i_1}, v_{i_2}, \ldots, v_{i_k} \} \subseteq V$ of people can communicate in a conference call if and only if their members use one and the same communication channel, i.e.  there is a hyperedge $e \in \cE$ such that $\{ v_{i_1}, v_{i_2}, \ldots, v_{i_k} \} \subseteq e$.
If we ask whether or not $v_{i_1}, v_{i_2}, \ldots, v_{i_k} $ can even communicate in a conference call after the breakdown of an arbitrarily chosen communication channel, then this question is equivalent to the problem of the existence of a hyperedge $\eie \in \eiE$ in the edge intersection hypergraph $EI(\cH)$ containing all these vertices, i.e.  $\{ v_{i_1}, v_{i_2}, \ldots, v_{i_k} \} \subseteq \eie$.

Note that there is a significant difference to the well-known notions of the {\em intersection graph} (cf. \cite{Nai}) or {\em edge intersection graph} (cf. \cite{Sku}) $G=(V(G), E(G))$  {\em of linear hypergraphs} $\cH = (V(\cH), \cE(\cH))$, since there we have $V(G) = \cE(\cH)$.

In \cite{Gol}, \cite{Bie} and  \cite{Cam} the same notation is used for so-called {\em edge intersection graphs of paths}, but there the authors consider paths in a given graph $G$ and the vertices of the resulting edge intersection graph correspond to these  paths in the original graph $G$.

Obviously, for certain hypergraphs $\cH$ the edge intersection hypergraph $EI(\cH)$  can be 2-uniform; in this case $EI(\cH)$ is a simple, undirected graph $G$. But in contrast to the intersection graphs or  edge intersection graphs mentioned above,   $G=EI(\cH)$ and $\cH$ have one and the same vertex set $V(G) = V(\cH)$.
Therefore we consistently use our notion "edge intersection hypergraph" also when this hypergraph is 2-uniform.

First of all, in Section 2 we investigate structural properties of edge intersection hypergraphs.

In Section 3 we consider 2-uniform edge intersection hypergraphs.
In doing so, a natural question arises.

\medskip
{\bf Problem 1.}
Which graphs are edge intersection hypergraphs?

\medskip
We will show that all but a few cycles, paths and trees are edge intersection hypergraphs of 3-uniform hypergraphs; the exceptional graphs have at most 6 vertices.
Whereas the proofs for cycles, paths and stars are simple, in the case of arbitrary trees we will make use of a special kind of induction.

\section{Some structural properties of edge intersection hypergraphs}

\begin{theorem}
\label{STA} \begin{enumerate}
	\item[(i)] For each linear hypergraph $\cH=(V,\cE)$ with $V\notin\cE$ there is a hypergraph $\cH'=(V,\cE')$ with $EI(\cH')=\cH$.
	\item[(ii)] Let $\cH=(V,\cE)$ be a hypergraph containing $e_1,e_2 \in \cE$ with $|e_1 \cap e_2|\ge 2$, $e_1 \not\subseteq e_2$, $ e_2 \not\subseteq e_1$ and $\cH'=(V,\cE')$ be a hypergraph with $\cH=EI(\cH')$. Then there is an $\tilde{e}\in \cE \setminus\{e_1,e_2\}$ with $e_1 \cap e_2 \subseteq \tilde{e}$.
	\item[(iii)] Not every hypergraph $\cH=(V,\cE)$ with $V \notin \cE$ is an edge intersection hypergraph of some hypergraph $\cH'=(V,\cE')$.
\end{enumerate}
\end{theorem}

\medskip
\begin{proof}
		\begin{enumerate}
		\item[(i)]  Choosing $\cE'=\cE \cup \{V\}$ we have $\cH=EI(\cH')$.
		\item[(ii)] There are vertices $v_1 \in e_1\setminus e_2$ and $v_2 \in e_2\setminus e_1$ and edges $e_1',e_1'',e_2',e_2'' \in \cE'$ with $e_1' \cap e_1''=e_1$  and $e_2' \cap e_2''=e_2$. Clearly
		$$\exists e^1 \in \{ e_1',e_1''\}:v_2\notin e^1 \wedge \exists e^2 \in \{e_2',e_2''\}:v_1 \notin e^2.$$
		W.l.o.g. let $e^1=e_1'$, $e^2=e_2'$. Then
		$$\tilde{e}:= e_1' \cap e_2' \supseteq (e_1' \cap e_1'') \cap (e_2' \cap e_2'')=e_1 \cap e_2.$$ Hence $\tilde{e} \in \cE$ and $e_1 \not\subseteq \tilde{e}$, $e_2 \not\subseteq \tilde{e}$.
		\item[(iii)] This follows from (ii); a minimal example is $\cH=(V,\cE)$ with $V=\{1,2,3,4\}$,\\ $\cE =\{ \{1,2,3\}, \{2,3,4\}\}$.\hqed
	\end{enumerate}
\end{proof}
\bigskip
Next we consider relations between edge intersection hypergraphs and several other classes of hypergraphs known from the literature:
The {\em competition hypergraph} $C \cH(D)=(V, \cE^C)$ of a digraph $D=(V,A)$ (see \cite{GST40}) has the edge set
$$ \cE^C=\{e \subseteq V \big| \, |e| \ge 2 \wedge \exists v \in V : e= \inN(v) \}. $$
The {\em double competition hypergraph} $DC \cH(D)=(V,\cE^{DC})$ of a digraph $D=(V,A)$ (see \cite{GST33}) has the edge set
$$\cE^{DC}
=\{e \subseteq V \big| \, |e| \ge 2 \wedge \exists v_1,v_2 \in V: e=\outN(v_1) \cap \inN(v_2)\}.$$
The {\em  niche hypergraph} $N\cH(D)=(V, \cE^N)$ of a digraph $D=(V,A)$ (see \cite{GST2016}) has the edge set
$$ 	\cE^N =\{e \subseteq V \big| \, |e| \ge 2 \wedge \exists v \in V: e=\inN(v) \lor e=\outN(v) \}. $$
Further, for technical reasons, we need the {\em common enemy hypergraph} $CE \cH(D)=(V, \cE^{CE})$ of a digraph $D=(V,A)$ with the edge set
$$ \cE^{CE}=\{e \subseteq V \big| \, |e| \ge 2 \wedge \exists v \in V : e= \outN(v) \}, $$ as well as the hypergraph  $\cH'(D)=(V, \cE')$  of a digraph $D=(V,A)$ with the edge set
$$ \cE'=\{e \subseteq V \big| \, |e| \ge 2 \wedge \exists v_1,v_2\in V : e=\inN(v_1)=\outN(v_2) \}.$$
The following theorem yields relations between these classes of hypergraphs.

\begin{theorem}
	Let $D=(V,A)$ be a digraph; then
	\begin{center}
		$EI(N\cH(D)) \cup \cH'(D) = DC\cH(D) \cup EI(C\cH(D)) \cup EI(CE\cH(D))$.
	\end{center}
\end{theorem}

\begin{proof}
	All hypergraphs have the vertex set $V$. For simplifying the logical expressions below, in the following let the symbol $e$ always denote a subset $e \subseteq V$ of cardinality at least 2.
Then we obtain
\begin{align*}
	& \; e\in (\cE^N)^\ei \cup \cE' \label{eq:3.5.2} \\
	\Leftrightarrow &\Big[\exists e_1,e_2 \in \cE^C \cup \cE^{CE}: e_1 \ne e_2 \wedge e=e_1 \cap e_2 \Big] \lor e \in \cE'  \\
	\Leftrightarrow &\Big[\exists e_1,e_2 \subseteq V: e_1\ne e_2 \; \wedge \; e=e_1 \cap e_2  \;  \wedge \; 
\Big(\exists v_1,v_2 \in V :\big(e_1=\inN(v_1) \wedge e_2=\inN(v_2)\big) \lor \\ 	
	&\phantom{  \Big[\exists e_1,e_2 \subseteq V: e_1\ne e_2 \; \wedge \; e=e_1 \cap e_2  \;  \wedge \;\Big(\exists v_1,v_2 \in V :} \;\big(e_1=\outN(v_1) \wedge e_2=\outN(v_2)\big)  \lor \\
	&\phantom{  \Big[\exists e_1,e_2 \subseteq V: e_1\ne e_2 \; \wedge \; e=e_1 \cap e_2  \;  \wedge \;\Big(\exists v_1,v_2 \in V :} \;\big(e_1=\inN(v_1) \wedge e_2=\outN(v_2)\big) \phantom{ \vee }\Big)\Big] \;  \lor \,  \\
	&  \; \Big[ \exists v_1,v_2 \in V: e=\inN(v_1)=\outN(v_2)\Big]  \\
	\end{align*}	

\begin{align*}
		\Leftrightarrow &\exists v_1,v_2 \in V \; \exists e_1,e_2 \subseteq V: e_1 \ne e_2 \; \; \wedge
\; \Big[\Big(e=e_1 \cap e_2 \wedge \big[(\, e_1=\inN(v_1) \wedge e_2=\inN(v_2)) \; \lor  \\
    & \phantom{ \exists v_1,v_2 \in V \; \exists e_1,e_2 \subseteq V: e_1 \ne e_2 \; \; \wedge \; \Big[\Big(e=e_1 \cap e_2 \wedge \big[} (e_1=\outN(v_1) \wedge e_2=\outN(v_2)) \; \lor \\
    & \phantom{ \exists v_1,v_2 \in V \; \exists e_1,e_2 \subseteq V: e_1 \ne e_2 \; \; \wedge \; \Big[\Big(e=e_1 \cap e_2 \wedge \big[}  (e_1=\inN(v_1) \wedge e_2=\outN(v_2))\phantom{\; \; \; \; }\big]\Big) \; \lor \\
    & \phantom{   \exists v_1,v_2 \in V \; \exists e_1,e_2 \subseteq V: e_1 \ne e_2 \; \; \wedge \; \Big[ ,\Big(e=e_1 \cap e_2 \wedge \big[(\, e_1=\inN(v_1) \wedge e} e= \inN(v_1)=\outN(v_2)\Big] \\[1ex]
	\Leftrightarrow
    &\exists v_1,v_2 \in V \, \exists e_1,e_2 \subseteq V:e=e_1 \cap e_2 \; \wedge \\
   	&  \big[e_1=\inN(v_1) \ne \inN(v_2)=e_2 \, \lor  \, e_1=\outN(v_1) \ne \outN(v_2)=e_2 \,
\lor \, (e_1=\inN(v_1) \wedge e_2=\outN(v_2)) \big]\\[1.5ex]
	\Leftrightarrow
    &\exists v_1,v_2 \in V \, \exists e_1,e_2 \subseteq V:e=e_1 \cap e_2 \; \wedge \\
    & \phantom{\exists v_1,v_2 \in V \, \exists e_1, }\big[\big(e_1,e_2 \in \cE^C  \wedge e_1 \ne e_2\big)  \; \lor  \; \big(e_1,e_2 \in \cE^{CE} \wedge e_1 \ne e_2\big) \; \lor \; \big(e_1 \in \cE^C \wedge e_2 \in \cE^{CE}  \big) \big] \\[1.5ex]
 	\Leftrightarrow
    &\big(\exists e_1,e_2\in \cE^{C}: e_1 \ne e_2 \; \wedge \;   e=e_1 \cap e_2 \big) \; \lor \;(\exists e_1,e_2 \in \cE^{CE}: e_1 \ne e_2 \wedge e=e_1 \cap e_2 \big) \; \\
     & \phantom{\big(\exists e_1,e_2\in \cE^{C}: e_1 \ne e_2 \; \wedge \;   e=e_1 \cap e_2 \big) \; } \lor \; \big(\exists v_1,v_2 \in V: e=\inN(v_1)\cap \outN(v_2)\big) \\[1ex]
	\Leftrightarrow  & \, e \in (\cE^{C})^\ei \; \lor \;   e \in (\cE^{CE})^\ei \; \lor \; e \in \cE^{DC} \\[1ex]
	\Leftrightarrow & \, e \in (\cE^{C})^\ei \; \cup \;  (\cE^{CE})^\ei \; \cup \; \cE^{DC}.
	\end{align*} \hqed
	
\end{proof}
\bigskip
A hypergraph $\cH=(V,\cE)$ has the {\em Helly property} if $$ \forall \cE' \subseteq \cE: (\forall e_1, e_2 \in \cE': e_1 \cap e_2 \not= \emptyset) \rightarrow \bigcap\limits_{e'\in \cE'}e' \ne \emptyset$$ (see Berge \cite{GST5}); next we show that the Helly property is hereditary for edge intersection hypergraphs.

\begin{theorem}
	If $\cH=(V, \cE)$ has the Helly property then  $EI(\cH) = (V, \eiE )$ has this property, too.
\end{theorem}

\begin{proof}
	Let $\eiE_S=\{e_1,...,e_t\} \subseteq \eiE$ with $t\ge 1$ and $e_i \cap e_j \ne \emptyset$ for $i,j \in \{1,...,t\}$. Clearly
	$$\forall i \in \{1,...,t\} \; \exists e_i',e_i'' \in \cE: e_i=e_i' \cap e_i'' \wedge e_i'\ne e_i''.$$
Let $\cE_S:=\{ e_1',...,e_t',e_1'',...,e_t''\}$. By $e_i \cap e_j \ne \emptyset$, for $i,j \in \{1,...,t\}$, we have $\bar{e} \cap \bar{\bar{e}} \ne \emptyset$ for arbitrary $\bar{e}, \bar{\bar{e}} \in \cE_S$ and the Helly property of $\cH$ yields
	$$\emptyset \not= \bigcap\limits_{\bar{e}\in \cE_S}\bar{e}=e_1' \cap...\cap e_t' \cap e_1''\cap...\cap e_t''=(e_1' \cap e_1'')\cap...\cap(e_t' \cap e_t'') = e_1 \cap...\cap e_t=\bigcap\limits_{e\in \eiE_S}e,$$ i.e. $EI(\cH)$ has the Helly property.\hqed
\end{proof}
\bigskip

From the definition of edge intersection hypergraphs it follows immediately that for $k \ge 1$
$$ \max\{ |e| \; \big| \; e \in \cE(EI^\title{k}(\cH))\} < \max\{ |e|  \; \big| \;  e \in \cE(EI^{k-1}(\cH)) \}. $$
Hence the $EI$-number $\kEI(\cH)$ is well defined. In the following we determine the edge intersection hypergraph and the $EI$-number $\kEI$ for some special classes of hypergraphs.
The strong {\em d-uniform hypercycle} $\hat{\cC}_{n}^d$ and the strong {\em d-uniform hyperpath} $\hat{\cP}_{n}^d$ both have the vertex set $\{v_1,...,v_n\}$ and the edge sets

\medskip
\qquad$ \cE(\hat{\cC}_n^d)=\{e_i=\{v_i,v_{i+1},...,v_{i+d-1} \}\big| i=1,...,n\} \quad \text{(indices taken modulo n)} $

\medskip
\qquad$ \cE(\hat{\cP}_n^d)=\{e_i=\{v_i,v_{i+1},...,v_{i+d-1}\} \big| i=1,...,n-d+1\}. $

\medskip
We consider only those strong d-uniform hypercycles $\hat{\cC}_{n}^d$ with $n \ge 2d-1$. This condition implies that for different edges $e_i,e_j \in \cE(\hat{\cC}_{n}^d)$ the intersection is empty or contains only  vertices being consecutive on the cycle, i.e. $e_i \cap e_j=\{v_s,v_{s+1},...,v_{s+t}\}$ for $s=1,...,n$ and $t=0,...,d-2$ (indices taken modulo n). For "small" cycles with $n<2d-1$ the edge intersection hypergraph is getting deep, because it contains edges of other types, too, and the following structural results, which are partly contained in the Bachelor Thesis \cite{ROB} of a student of the second author, are not true.

\begin{theorem}
	Let $\hat{\cC}_{n}^d$ and $\hat{\cP}_{n}^d$ be a strong d-uniform hypercycle and a strong d-uniform hyperpath, respectively.
	\begin{enumerate}
	 \item[(i)] $EI^k(\hat{\cC}_n^d)=\hat{\cC}^{d-k}_n \cup \hat{\cC}^{d-k-1}_n \cup ... \cup \hat{\cC}^2_n $ for $d \ge 3$, $n \ge 2d-1$ and $k=1,...,d-2.$
	 \item[(ii)]  $\kEI(\hat{\cC}_n^d)=d-1$ for $d \ge 2$ and $n \ge 2d-1$.
	 \item[(iii)] $\kEI(\hat{\cP}_n^d)=\begin{cases}
	 d-1 & \text{for } d \ge 2 \text{ and } n\ge 2d-1, \\
	 n-d+1 & \text{for }d \ge 2 \text{ and } n < 2d-1. \\
	 \end{cases}$
	\end{enumerate}
\end{theorem}

\begin{proof}
	\begin{enumerate}
		\item[(i)] In strong d-uniform hypercycles $\hat{\cC}_{n}^d$ with $n \ge 2d-1$ there are intersections of cardinalities at least two between the edges $e_i,e_{i+1},...,e_{i+d-2}$; $i=1,...,n$ (indices taken modulo n). Hence $EI(\hat{\cC}_{n}^d)$ contains the following edges
(see Figure~\ref{abb:hypercycle}):
		$$ e_{j,i}:= e_i \cap e_{i+j}=\{v_{i+j},...,v_{i+d-1}\},$$
with $ i=1,...,n$ and $ j=1,...,d-2$ (indices taken modulo n).
		This yields $EI(\hat{\cC}_n^d)=\hat{\cC}^{d-1}_n \cup \hat{\cC}^{d-2}_n \cup ... \cup \hat{\cC}^2_n$, i.e. by using the $EI$-operator the maximum edge cardinality decreases by one. For the $k$-th iteration we obtain
		$$EI^k(\hat{\cC}_n^d)=\hat{\cC}^{d-k}_n \cup... \cup \hat{\cC}^2_n \, , \; \; k=1,...,d-2.$$
			\begin{figure}[h]
				\makeatletter
				\renewcommand{\@thesubfigure}{}
				\makeatother
			\centering
			\subfigure[$\hat{\cC}_{10}^4$ ]{\includegraphics[width=5.8cm]{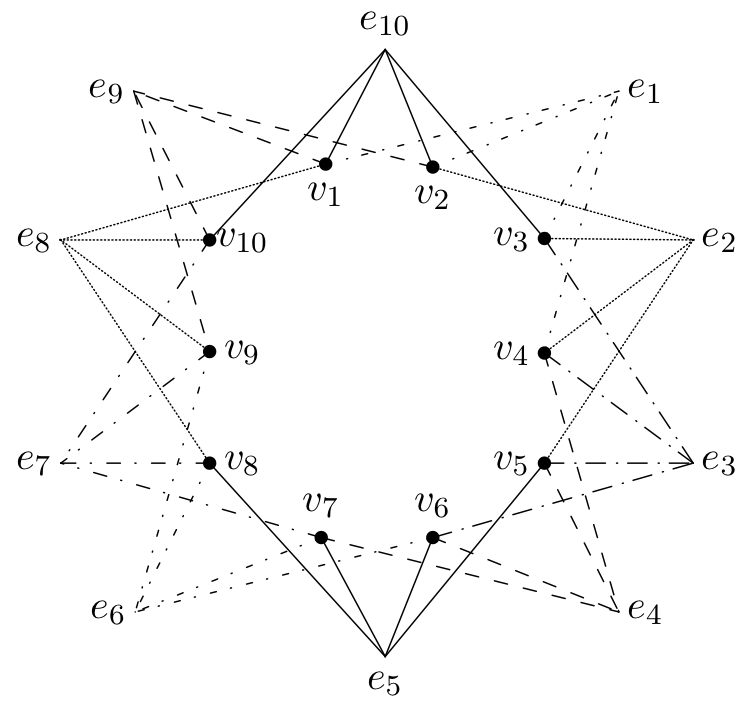}}
			\hfill
			\subfigure[$EI^1(\hat{\cC}_{10}^4)=\hat{\cC}^{3}_{10} \cup \hat{\cC}^{2}_{10} $ ]{\includegraphics[width=5.5cm]{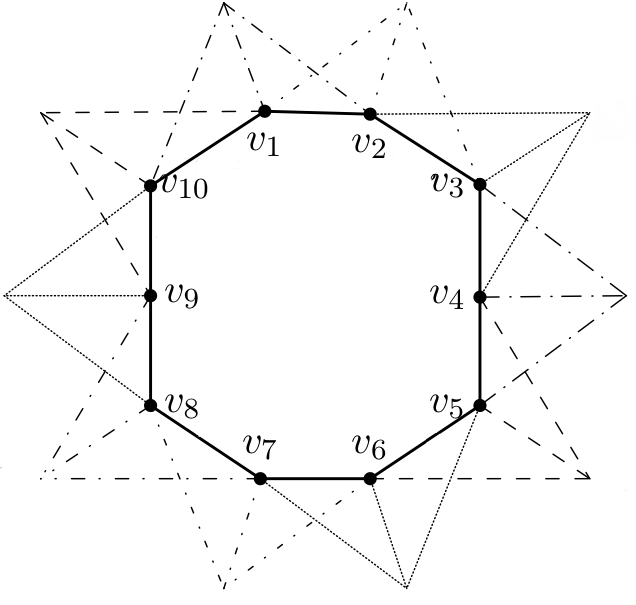}}
			\hfill
			\subfigure[$EI^2(\hat{\cC}_{10}^4)=\hat{\cC}^{2}_{10} $]{\includegraphics[width=3.3cm]{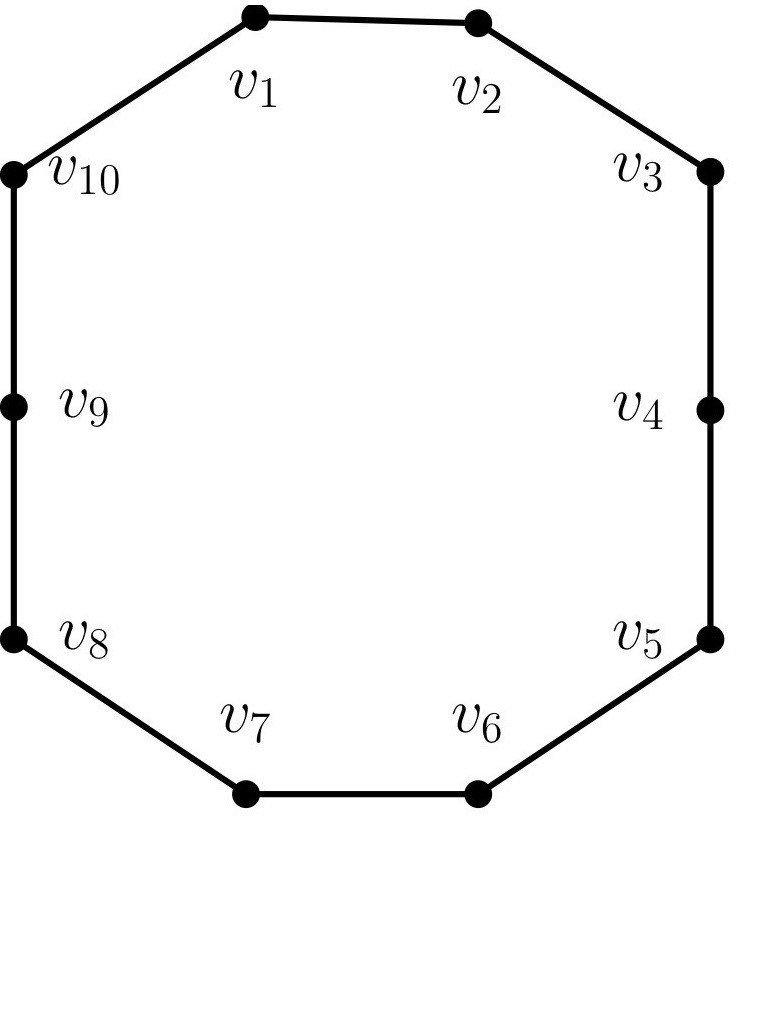}}
			\caption{The strong 4-uniform hypercycle $\hat{\cC}_{10}^4$ and the corresponding edge intersection hypergraphs.}
			\label{abb:hypercycle}
		\end{figure}

\vspace{-7mm}
		\item[(ii)] The case $d=2$ is trivial; for $d \ge 3$ it follows with (i) that $EI^{d-2}(\hat{\cC}_n^d)=\hat{\cC}_n^2=C_n$ and hence $\kEI(\hat{\cC}_n^d)=d-1$.
		\item[(iii)] The result is trivial for $d=2$ in both cases; in the following we assume $d \ge 3$.\\
\end{enumerate}
\vspace{-5mm}
For $n \ge 2d-1$ we have $|e_1 \cap e_{n-d+1}| \le 1$, i.e. the intersection of the first edge and of the last edge of $\hat{\cP}_{n}^d$ does not generate an edge in $EI(\hat{\cP}_{n}^d)$. The edges of $EI(\hat{\cP}_{n}^d)$ are generated by the following intersections
(see Figure~\ref{abb:hyperpath}):
		\begin{align*}
		e_{1,i}:= e_i \cap e_{i+1} &=\{v_i,v_{i+1},...,v_{i+d-1}\} \cap \{v_{i+1},v_{i+2},...,v_{i+d}\}\\ &= \{v_{i+1},...,v_{i+d-1}\} \qquad \qquad \qquad  \; \, \text{for} \quad i=1,...,n-d, \\
		e_{2,i}:= e_i \cap e_{i+2} &=\{v_i,v_{i+1},...,v_{i+d-1}\} \cap \{v_{i+2},v_{i+3},...,v_{i+d+1}\} \\&= \{v_{i+2},...,v_{i+d-1}\}  \qquad \qquad \qquad  \; \; \text{for} \quad i=1,...,n-d-1,\\
		& \vdots \\
		e_{d-2,i}:= e_i \cap e_{i+d-2} &=\{v_i,v_{i+1},...,v_{i+d-1}\} \cap \{v_{i+d-2},v_{i+d-1},...,v_{i+d+1}\} \\&= \{v_{i+d-2},v_{i+d-1}\} \qquad \qquad \qquad  \; \; \; \text{for} \quad i=1,...,n-2d+3.\\
		\end{align*}
		Hence $EI(\hat{\cP}_{n}^d)$ has edges of cardinalities  $d-1,d-2,..,2$ and the edge set		 \begin{align*}
		\cE(EI(\hat{\cP}_n^d))
		=& \; \{e_{1,1},...,e_{1,n-d},e_{2,i},...,e_{2,n-d-1},e_{d-2,1},...,e_{d-2,n-2d+3}\}\\
		=& \; \cE(\hat{\cP}_n^{d-1})  \setminus  \{ \{ v_1,...,v_{d-1}\},\{ v_{n-d+2},...,v_{n}\} \} \quad  \, \cup \\
		 &\; \cE(\hat{\cP}_n^{d-2})  \setminus  \{ \{ v_1,...,v_{d-2}\},\{ v_2,...,v_{d-1}\},\{ v_{n-d+2},...,v_{n-1}\}, \\& \phantom{\cE(\hat{\cP}_n^{d-2})  \setminus  \{}
		 \qquad \qquad \qquad \qquad \qquad \qquad \quad \; \{ v_{n-d+3},...,v_{n}\} \}   \cup \ldots \cup \, \\
		 & \cE(\hat{\cP}_n^2)  \setminus  \{ \{ v_1,v_{2}\},...,\{ v_{d-2},v_{d-1}\},\{ v_{n-d+2},v_{n-d+3}\},...,\{ v_{n-1},v_{n}\} \}.
\end{align*}
		
			\begin{figure}[h!]
				\makeatletter
				\renewcommand{\@thesubfigure}{}
				\makeatother
			\centering
			\subfigure[$\hat{\cP}_{10}^4 $ ]{\includegraphics[width=9.35cm]{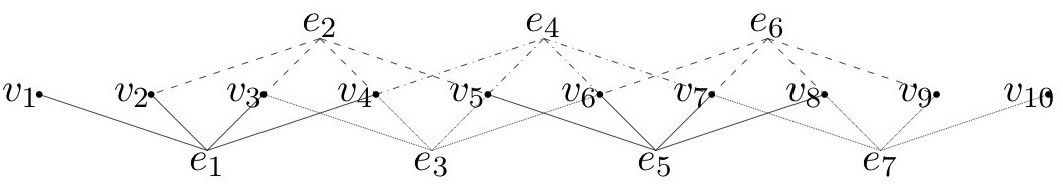}} 
			\vfill
			\subfigure[$EI^1(\hat{\cP}_{10}^4) $]{\includegraphics[width=9.35cm]{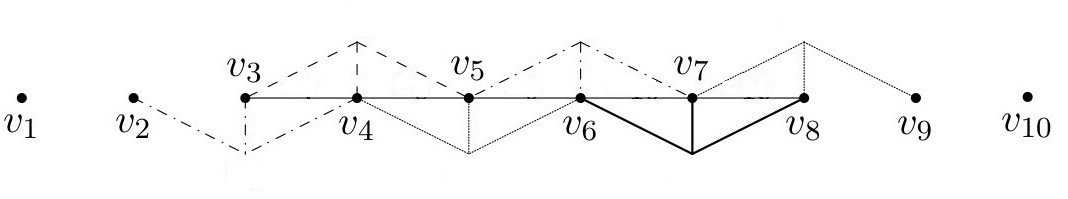}} 
			\vfill
			\subfigure[$EI^2(\hat{\cP}_{10}^4)$ ]{\includegraphics[width=9.35cm]{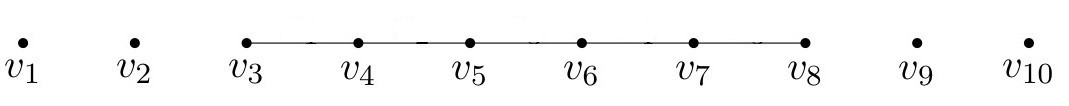}} 
			\vfill
			\subfigure[$EI^3(\hat{\cP}_{10}^4)$ ]{\includegraphics[width=9.35cm]{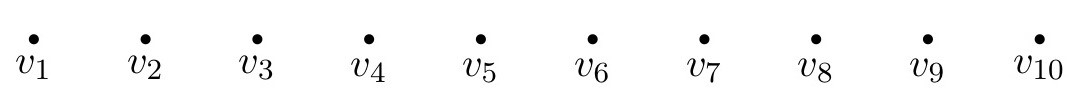}}
			\caption{The strong 4-uniform hyperpath $\hat{\cP}_{10}^4$ and the corresponding edge intersection hypergraphs.}
			\label{abb:hyperpath}
		\end{figure}
	
		The reapplication of the $EI$-operator yields a hypergraph without the edges of maximum cardinality $(d-1)$, while all other edges of $EI(\hat{\cP}_n^d)$ remain (because they are contained in the edges of cardinality $(d-1)$). After $(d-2)$ iterations we obtain $EI^{d-2}(\hat{\cP}_n^d)=P_{n-2d+4}\cup I_{2d-4}$, where $I_t$ denotes a set of $t$ isolated vertices; hence $\kEI(\hat{\cP}_n^d)=d-1$.\\
		
		For $n<2d-1$ we have $|e_1 \cap e_{n-d+1}| \ge 2$, i.e. the intersection of the first and the last edge of $\hat{\cP}_n^d$ generates in $EI(\hat{\cP_n^d})$ the edge of minimum cardinality $(2d-n)$. All edges of $EI(\hat{\cP_n^d})$ are generated by the following intersections
(see Figure~\ref{abb:hyperpath7_5}):
			\begin{align*}
		e_{1,i}:= e_i \cap e_{i+1} &=\{v_i,v_{i+1},...,v_{i+d-1}\} \cap \{v_{i+1},v_{i+2},...,v_{i+d}\}\\ &= \{v_{i+1},...,v_{i+d-1}\} \qquad\qquad \qquad \qquad  \text{for} \quad i=1,...,n-d, \\
		e_{2,i}:= e_i \cap e_{i+2} &=\{v_i,v_{i+1},...,v_{i+d-1}\} \cap \{v_{i+2},v_{i+3},...,v_{i+d+1}\} \\&= \{v_{i+2},...,v_{i+d-1}\}  \qquad \qquad \qquad \qquad  \text{for} \quad i=1,...,n-d-1,\\
		& \vdots \\
		e_{n-d-1,i}:= e_i \cap e_{i+n-d-1} &=\{v_i,v_{i+1},...,v_{i+d-1}\} \cap \{v_{i+n-d-1},...,v_{i+n-2}\} \\&= \{v_{i+n-d-1},...,v_{i+d-1}\} \qquad \qquad \qquad  \text{for} \quad i=1,2, \\
		e_{n-d,1}:= e_1 \cap e_{n-d+1} &=\{v_1,v_{2},...,v_{d}\} \cap \{v_{n-d+1},...,v_{n}\} \\&= \{v_{n-d+1},...,v_{d}\}. \\
		\end{align*}
		Hence $EI(\hat{\cP}_{n}^d)$ has edges of cardinalities  $d-1,d-2,..,2d-n$ and the edge set
		\begin{align*}
		\cE(EI(\hat{\cP}_n^d))
		=&\; \{e_{1,1},...,e_{1,n-d},e_{2,1},...,e_{2,n-d-1},e_{n-d-1,1},e_{n-d-1,2},e_{n-d,1}\}\\
		=&\; \cE(\hat{\cP}_n^{d-1})  \setminus  \{ \{ v_1,...,v_{d-1}\},\{ v_{n-d+2},...,v_{n}\} \} \quad \,  \cup \\
&\; \cE(\hat{\cP}_n^{d-2})  \setminus  \{ \{ v_1,...,v_{d-2}\},\{ v_2,...,v_{d-1}\},\{ v_{n-d+2},...,v_{n-1}\},\\
& \qquad \qquad \qquad \qquad \qquad \qquad \qquad \qquad \qquad  \quad \{ v_{n-d+3},...,v_{n}\} \}  \; \cup \ldots \cup \\
&\; \cE(\hat{\cP}_n^{2d-n})  \setminus  \{ \{v_1,...,v_{2d-n}\},...,\{ v_{n-d},...,v_{d-1}\},\\
& \qquad \qquad \quad \; \; \; \{ v_{n-d+2},...,v_{d+1}\},...,\{ v_{2(n-d)+1},...,v_{n}\} \}.
		\end{align*}
			\begin{figure}[!h]
			\centering
				\makeatletter
			\renewcommand{\@thesubfigure}{}
			\makeatother
			 \subfigure[$\hat{\cP}_{7}^5$]{\includegraphics[width=7.5cm]{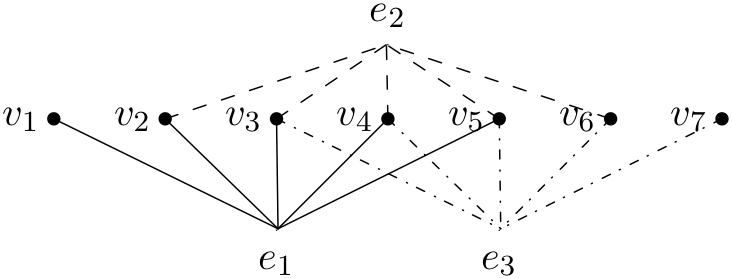}}
			\hfill
			\subfigure[$EI^1(\hat{\cP}_{7}^5) $]{\includegraphics[width=7.5cm]{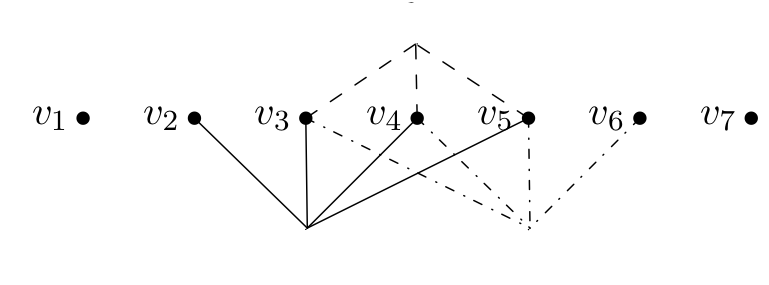}}
			\vfill
			\subfigure[$EI^2(\hat{\cP}_{7}^5) $]{\includegraphics[width=7.5cm]{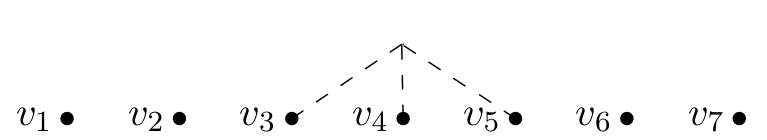}}
			\hfill
			\subfigure[$EI^3(\hat{\cP}_{7}^5) $]{\includegraphics[width=7.5cm]{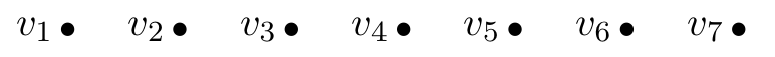}}
			\caption{The strong 5-uniform hyperpath $\hat{\cP}_7^5$ and the corresponding edge intersection hypergraphs. }
			\label{abb:hyperpath7_5}
		\end{figure}
	
Again, the reapplication of the $EI$-operator yields a hypergraph without the edges of maximum cardinality $(d-1)$, while all the other edges remain. After $(n-d)$ iterations we obtain $EI^{n-d}(\hat{\cP}_n^d)=\tilde{e} \cup I_{2(n-d)}$, where $\tilde{e}=\{v_{n-d+1},...,v_d\}$ with cardinality $|\tilde{e}|=2d-n$, hence $\kEI(\hat{\cP}_n^d)=n-d+1$.\hqed
\end{proof}
\smallskip

For $n \ge 5$, $d=3$ and $k=1$, Theorem 4(i) provides the following.

\begin{corollary}
For $n \ge 5$ the cycle $C_n$ is an edge intersection hypergraph of a 3-uniform \\[0.3ex] hypergraph, namely $ C_n = EI(\hat{\cC}_n^3)$.
\end{corollary}

Berge \cite{GST5} generalized  the complete graph $K_n$ by the definition of the {\em complete d-uniform hypergraph} $\cK_n^d$ as follows:
$$ V(\cK_n^d)=\{v_1,...,v_n\}\, , \; \; \cE(\cK_n^d)=\{T\subseteq V(\cK_n^d) \; \big| \; |T|=d \}.$$

\begin{theorem}
	Let $\cK_n^d$ be a complete d-uniform hypergraph with $n-1 \ge d \ge 3$.
	\begin{enumerate}
		\item[(i)] $EI^k(\cK_n^{d}) =\cK_n^{d-k} \cup \cK_n^{d-k-1} \cup ...\cup \cK_n^{t_k}$ for $1 \le k \le d-2$,\\[0.5ex]
		where $t_k:=\max\{2, 2^k(d-n)+n\}$ for $1 \le k \le d-2.$
		\item[(ii)] $\kEI(\cK_n^d)=d-1$ for $d \ge 2$.
	\end{enumerate}
\end{theorem}
	\begin{proof}
		\begin{enumerate}
			\item[(i)]  The intersections (with cardinality of a least two) of edges in $\cK_n^d$ are all subsets $T \subseteq V(\cK)_n^d$  with cardinalities in the range $d-1 \ge |T| \ge t_1 = \max\{2,2d-n\}$; hence
			$$EI(\cK_n^{d})=\cK_n^{d-1} \cup \cK_n^{d-2} \cup ...\cup \cK_n^{t_1}.$$
			Using induction, the reapplication of the $EI$-operator to $EI^k(\cK_n^d)$ yields all subsets $T \subseteq V(\cK_n^d)$ of cardinalities in the range $$d-(k+1)\ge|T| \ge \max\{2,2(2^k(d-n)+n)-n\}=\max\{2, 2^{k+1}(d-n)+n\}. $$
			\item[(ii)] From (i) we know that $EI^{d-2}(\cK_n^d)=\cK_n^2=K_n$, hence $\kEI(\cK_n^d)=d-1$.\hqed
		\end{enumerate}
	\end{proof}

\section{Trees as edge intersection hypergraphs}

In the following, for the trees up to 8 vertices we often use the notations $T1, T2, \ldots, T48$ corresponding to \cite{RW}.
Moreover, for shortness we will conveniently write $ij$ instead of $\{i, j \}$ and $ijk$ instead of $\{ i,j,k \}$ for edges and hyperedges, respectively.

The main result of the section is that all but seven exceptional trees are edge intersection hypergraphs of 3-uniform hypergraphs. The exceptional trees have at most 6 vertices.

\begin{theorem}
All trees but $T2=P_2, T3=P_3, T5=P_4, T8=P_5, T14=P_6, T7$ and $T12$ are edge intersection hypergraphs of a 3-uniform hypergraph $\cH$.
\end{theorem}

\medskip
The proof will be done by induction. The  induction basis includes the investigation of all 48 trees having at most 8 vertices (see Lemma 1 -- Lemma 3 below). Note that this set of trees contains the  seven exceptional cases mentioned above.

The  inductive step will make use of the deletion of a (shortest) so-called {\em leg} in a tree.

\pagebreak

\subsection{Induction basis}
\label{Sec_basis}

Above all, for two simple classes of trees, namely paths and stars, we easily obtain a first result.

\begin{lemma}
\begin{enumerate}
\item[(i)]
For $n \ge 3$, the star $K_{1,n}$ is an edge intersection hypergraph of a 3-uniform hypergraph.
\item[(ii)]
For $n=1$ and for $n\ge7$, the path $P_n$ is an edge intersection hypergraph of a 3-uniform hypergraph.
\end{enumerate}
\end{lemma}
\begin{proof}
\begin{enumerate}
\item[(i)] Let $K_{1,n} = ( V, E)$ with $V=\{ 1,2, \ldots, n, n+1 \}$,
$E=\{ \{ 1,2 \}, \{ 1,3 \}, \ldots, \{1,n \}, \{1, n+1 \} \}$ and $\cH = ( V, \cE)$ with
$\cE = \{ \{ 1,2,3 \}, \{ 1,3,4 \}, \ldots, \{ 1, n, n+1 \}, \{ 1, n+1, 2 \} \}$.
Then $K_{1,n} = EI(\cH)$.

\item[(ii)] Let $n \ge 7$ and $P_n =  \hat{\cP}_n^2 =(V, E)$; for simplicity we identify the vertices $v_i \in V$ with their indices: $v_i = i$.

With $\cH = ( V, \cE )$, where $\cE = \{ \{ 1,2,3 \}, \{2,3,4 \}, \ldots, \{ n-2, n-1, n \}, \{ 1,2, n-2 \}, \{n-1, n, 3 \} \}$, we have $P_n = EI(\cH)$. \hqed
\end{enumerate} \end{proof}

Now we discuss the seven exceptional trees.

\begin{lemma}
$T2=P_2, T3=P_3, T5=P_4, T8=P_5, T14=P_6, T7$ and $T12$ are not edge intersection hypergraphs of a 3-uniform hypergraph.
\end{lemma}

\begin{proof}
In the following, we give the most laborious proofs for $P_6$ and $T12$, respectively. All other cases can be shown  in a similar, but easier way.
\begin{enumerate}

\item[$(P_6)$]
Note that we have $E(P_6) = \{ \{1,2 \}, \{ 2,3 \}, \{ 3,4 \}, \{ 4,5 \}, \{5,6 \} \}$.

For any graph $G$, to generate the edges of $G = EI(\cH)$, in $\cH$ we need only hyperedges $e$ which contain at least two of the adjacent vertices of $G$. Such hyperedges will be called {\em useful hyperedges}.

In the present case, we have $G=P_6$ and a useful hyperedge $e$ has to fulfil  $\{i, i+1 \} \subset e$, where $i \in \{1,2,\ldots, 5 \}$. Therefore, the useful hyperedges which may occur in $\cH = (V, \cE)$ are the following:

123, 124, 125, 126 -- to generate $12$ in $P_6$;

234, 235, 236 -- to generate $23$  (should the occasion arise, in connection with 123);

134, 345, 346 -- to generate $34$  (\ldots with 234);

145, 245, 456 -- to generate $45$  (\ldots with 345);

156, 256, 356 -- to generate $56$  (\ldots with 456).

Clearly, each of the edges $\{i, i+1 \}$ in $P_6$ is contained in exactly four useful hyperedges and at least two of these hyperedges have to appear in $\cH$ to generate $\{i, i+1 \}$ in $EI(\cH) = P_6$.

By case distinction, we discuss the possible combinations of useful hyperedges in $\cH$ and will obtain a contradiction (abbreviated by the symbol $\lightning$) in every case. We have six possibilities to generate $12 \in \cE(EI(\cH)) = E(P_6)$.

\begin{itemize}
\item[(a)] If $123, 124 \in \cE$, then $134 \notin \cE$ (otherwise $13 \in E(P_6)$ $\lightning$) and $234 \notin \cE$ (otherwise $24 \in E(P_6)$ $\lightning$). Therefore, $345, 346 \in \cE$ in order to generate $34  \in E(P_6)$.

    On the other hand,  $145 \notin \cE$ (otherwise $14 \in E(P_6)$ $\lightning$) and $245 \notin \cE$ (otherwise $24 \in E(P_6)$ $\lightning$). Hence,  $456 \in \cE$ in order to generate $45  \in E(P_6)$.

    This includes $346 \cap 456 = 46 \in \cE(EI(\cH)) = E(P_6)$ $\lightning$.

\item[(b)] If $123, 125 \in \cE$, then $145 \notin \cE$ (otherwise $15 \in E(P_6)$ $\lightning$) and $245 \notin \cE$ (otherwise $25 \in E(P_6)$ $\lightning$).
    Therefore, $345, 456 \in \cE$ in order to generate $45  \in E(P_6)$.

    On the other hand,  $156 \notin \cE$ (otherwise $15 \in E(P_6)$ $\lightning$) and $256 \notin \cE$ (otherwise $25 \in E(P_6)$ $\lightning$). Hence,  $356 \in \cE$ in order to generate $56  \in E(P_6)$.

   This includes $345 \cap 356 = 35 \in \cE(EI(\cH)) = E(P_6)$ $\lightning$.

\item[(c)]  If $123, 126 \in \cE$, then $156 \notin \cE$ (otherwise $16 \in E(P_6)$ $\lightning$) and $256 \notin \cE$ (otherwise $26 \in E(P_6)$ $\lightning$).
    Therefore, $356, 456 \in \cE$ in order to generate $56  \in E(P_6)$.

    On the other hand,  $134 \notin \cE$ (otherwise $13 \in E(P_6)$ $\lightning$), $345 \notin \cE$ (otherwise $35 \in E(P_6)$ $\lightning$) and $346 \notin \cE$ (otherwise $46 \in E(P_6)$ $\lightning$). Hence,  $34 \notin \cE(EI(\cH)) = E(P_6)$ $\lightning$.

\item[(d)]  If $124, 125 \in \cE$, then $123 \notin \cE$ (because of (a), (b)), $234 \notin \cE$ (otherwise $24 \in E(P_6)$ $\lightning$) and  $235 \notin \cE$ (otherwise $25 \in E(P_6)$ $\lightning$).

    Consequently,  $23 \notin \cE(EI(\cH)) = E(P_6)$ $\lightning$.

\item[(e)]  If $124, 126 \in \cE$, then $123 \notin \cE$ (because of (a), (c)), $234 \notin \cE$ (otherwise $24 \in E(P_6)$ $\lightning$) and  $236 \notin \cE$ (otherwise $26 \in E(P_6)$ $\lightning$).

    Consequently,  $23 \notin \cE(EI(\cH)) = E(P_6)$ $\lightning$.

\item[(f)]  If $125, 126 \in \cE$, then $123 \notin \cE$ (because of (b), (c)), $235 \notin \cE$ (otherwise $25 \in E(P_6)$ $\lightning$) and  $236 \notin \cE$ (otherwise $26 \in E(P_6)$ $\lightning$).

    Consequently,  $23 \notin \cE(EI(\cH)) = E(P_6)$ $\lightning$.

\end{itemize}
 This implies that 12 cannot be generated in $EI(\cH) = P_6$ $\lightning$. Therefore, $P_6$ is not an edge  intersection hypergraph of a 3-uniform hypergraph.

\item[($T12$)] The most laborious  case is the graph $T12 = ( V, E)$ with $E= \{ \{ 1,2\}, \{2,3 \},\{3,4 \}, \{ 2,5 \}, \{ 5,6 \} \}$.

The useful hyperedges which may occur in $\cH = (V, \cE)$ are the following:

123, 124, 125, 126 -- to generate $12$ in $T12$;

234, 235, 236 -- to generate $23$  (should the occasion arise, in connection with 123);

134, 345, 346 -- to generate $34$ (\ldots with 234);

245, 256 -- to generate $25$  (\ldots with 125 or 235);

156, 356, 456 -- to generate $56$  (\ldots with 256).


Again, we discuss the  combinations of useful hyperedges in $\cH$ and will obtain a contradiction  in every case. Obviously, we have six possibilities to generate $12 \in \cE(EI(\cH)) = E(T12)$ and in some cases several subcases have to be investigated.

\begin{itemize}
\item[(a)] If $123, 124 \in \cE$, then $134 \notin \cE$ (otherwise $13 \in E(T12)$ $\lightning$) and $234 \notin \cE$ (otherwise $24 \in E(T12)$ $\lightning$). Therefore, $345, 346 \in \cE$ in order to generate $34  \in E(T12)$.

    Hence,  $235 \notin \cE$ (otherwise $35 \in E(T12)$ $\lightning$). In order to generate $23  \in E(T12)$ it follows $236 \in \cE$.

    This includes $236 \cap 346 = 36 \in \cE(EI(\cH)) = E(T12)$ $\lightning$.

\item[(b)] If $123, 125 \in \cE$, then $134 \notin \cE$ (otherwise $13 \in E(T12)$ $\lightning$) and $156 \notin \cE$ (otherwise $15 \in E(T12)$ $\lightning$).

    In order to generate the edge $23 \in E(T12)$, the three subcases (b1), (b2) and (b3) have to be considered.

\begin{itemize}
\item[(b1)]If $ 234 \in \cE$, then for $34 \in E(T12)$ we need $345 \in \cE$ or $346 \in \cE$.

Assume, $345 \in \cE$. Then $356 \notin \cE$ (otherwise $35 \in E(T12)$ $\lightning$) and $456 \notin \cE$ (otherwise $45 \in E(T12)$ $\lightning$). Together with $156 \notin \cE$ (see (b) above), we obtain $56 \notin E(T12)$ $\lightning$.

So assume $346 \in \cE$. Then $356 \notin \cE$ (otherwise $36 \in E(T12)$ $\lightning$) and $456 \notin \cE$ (otherwise $46 \in E(T12)$ $\lightning$). As above,  $56 \notin E(T12)$ $\lightning$.

Consequently, (b1) cannot occur.

\item[(b2)] If $ 235 \in \cE$, then $ 345 \notin \cE$ follows (otherwise $35 \in E(T12)$ $\lightning$). Since $134 \notin \cE$ (see (b)) and $ 234 \notin \cE$ (see (b1)), we easily get $34 \notin E(T12)$ $\lightning$.\\[0.5ex]
\hspace*{-9mm} So the only possibility in case (b) would be the next one.
\\[-1ex]

\item[(b3)] Let $ 236 \in \cE$.
We have $ 256 \notin \cE$ (otherwise $26 \in E(T12)$ $\lightning$) and, additionally, since (b2) is impossible, also $235 \notin \cE$.  Hence for $25 \in E(T12)$ we need $245 \in \cE$.

Since $156 \notin E(T12)$ (see at the beginning of (b)) for $56 \in E(T12)$ necessarily $356, 456 \in \cE$. This provides $245 \cap 456 = 45  \in \cE(EI(\cH)) = E(T12)$ $\lightning$.

\end{itemize}

Thus Case (b) cannot occur.

\item[(c)]  If $123, 126 \in \cE$, then $125 \notin \cE$ (because of (b)), $156 \notin \cE$ (otherwise $16 \in E(T12)$ $\lightning$) and $256 \notin \cE$ (otherwise $26 \in E(T12)$ $\lightning$).

    Therefore, it follows $235, 245 \in \cE$ in order to generate $25  \in E(T12)$ as well as $356, 456 \in \cE$ in order to generate $56  \in E(T12)$.

    But then $235 \cap 356 = 35  \in \cE(EI(\cH)) = E(T12)$ $\lightning$.

\item[(d)]  If $124, 125 \in \cE$, then $123 \notin \cE$ (because of (a), (b)), $134 \notin \cE$ (otherwise $14 \in E(T12)$ $\lightning$) and  $234 \notin \cE$ (otherwise $24 \in E(T12)$ $\lightning$).

    Consequently,  we need $345, 346 \in \cE$ in order to generate $34  \in E(T12)$. Thus $236 \notin \cE$ (otherwise $36 \in E(T12)$ $\lightning$). Together with $123, 234 \notin \cE$  we obtain $23 \notin \cE(EI(\cH)) = E(T12)$ $\lightning$.

\item[(e)]  If $124, 126 \in \cE$, then $123 \notin \cE$ (because of (a), (c)), $234 \notin \cE$ (otherwise $24 \in E(T12)$ $\lightning$) and  $236 \notin \cE$ (otherwise $26 \in E(T12)$ $\lightning$).

    Consequently,  $23 \notin \cE(EI(\cH)) = E(T12)$ $\lightning$.

\item[(f)]  If $125, 126 \in \cE$, then $123 \notin \cE$ (because of (b), (c)) and $236 \notin \cE$ (otherwise $26 \in E(T12)$ $\lightning$). So we need $234, 235 \in \cE$ in order to generate $23 \in E(T12)$.
    On the other hand, $156 \notin \cE$ (otherwise $15 \in E(T12)$ $\lightning$)
     and  $256 \notin \cE$ (otherwise $26 \in E(T12)$ $\lightning$). Hence necessarily $356, 456 \in \cE$ in order to generate $56 \in E(T12)$.

   This leads to $235 \cap 356 = 35  \in \cE(EI(\cH)) = E(T12)$ $\lightning$.

\end{itemize}
 This implies that 12 cannot be generated in $EI(\cH) = T12$ $\lightning$. Therefore, $T12$ is not an edge intersection hypergraph of a 3-uniform hypergraph. \hqed
\end{enumerate}

\end{proof}

\begin{lemma}
All trees with at most eight vertices are edge intersection hypergraphs of a 3-uniform hypergraph, but $T2=P_2, T3=P_3, T5=P_4, T8=P_5, T14=P_6, T7$ and $T12$.
\end{lemma}

\begin{proof}
Because of Lemma 1 and Lemma 2 for the trees $T1 - T9$, $T12$, $T14$, $T15$, $T25$, $T26$ and $T48$ there is nothing to show.

For the remaining 33 trees $Tn = (V_n, E_n)$ in each case we give the edge set $E_n$ and the set of hyperedges $\cE_n$ of a 3-uniform hypergraph $\cH_n = (V_n, \cE_n)$ with $Tn = EI(\cH_n)$.
The verification of $\cE(EI(\cH_n)) = E(Tn)$ can be done  by hand for all $n$
 or by computer, e.g. using the computer algebra system MATHEMATICA$^{\textregistered}$ with the function
\\[0.8ex]
\indent
$EEI[eh\_] :=
 Complement[
  Select[Union[Flatten[Outer[Intersection, eh, eh, 1], 1]],$
\vspace{-1ex}
\begin{flushright} $   Length[\#] > 1 \&], eh],$ \end{flushright}

\smallskip
\noindent
where the argument $eh$ has to be the list of the hyperedges of $\cH$ in the form $\{\{a,b,c \}, \ldots, \{x,y,z \} \}$. Then $EEI[eh]$ provides the list of the hyperedges of $EI(\cH)$.

\medskip
\noindent
$n=6$ vertices:

\smallskip
$E_{10} = \{ 12, 23, 34, 35, 36 \}$,
$\cE_{10} = \{ 123, 124, 235, 236, 345, 346 \}.$

$E_{11} = \{ 12, 23, 24, 45, 46 \}$,
$\cE_{11} = \{ 123, 124, 234, 245, 246, 456 \}.$

$E_{13} = \{ 12, 23, 34, 45, 46 \}$,
$\cE_{13} = \{ 123, 125, 234, 345, 346, 456 \}.$

\pagebreak
\noindent
$n=7$ vertices:

\smallskip
$E_{16} = \{ 12, 23, 34, 35, 36, 37 \}$,
$\cE_{16} = \{ 123, 125, 234, 237, 345, 356, 367 \}.$

$E_{17} = \{ 12, 23, 24, 25, 56, 57 \}$,
$\cE_{17} = \{ 123, 124, 234, 256, 257, 567 \}.$

$E_{18} = \{ 12, 23, 24, 25, 56, 67 \}$,
$\cE_{18} = \{ 123, 124, 167, 234, 245, 256, 567 \}.$

$E_{19} = \{ 12, 23, 34, 35, 36, 67 \}$,
$\cE_{19} = \{ 123, 124, 235, 236, 345, 346, 367, 567 \}.$

$E_{20} = \{ 12, 23, 24, 45, 56, 57 \}$,
$\cE_{20} = \{ 123, 124, 234, 456, 457, 567 \}.$

$E_{21} = \{ 12, 23, 34, 36,  45, 47 \}$,
$\cE_{21} = \{ 123, 125, 234, 236, 345, 346, 347, 457 \}.$

$E_{22} = \{ 12, 23, 24, 45,  56, 67 \}$,
$\cE_{22} = \{ 123, 124, 167, 234, 245, 456, 567 \}.$

$E_{23} = \{ 12, 23, 34, 45,  46, 57 \}$,
$\cE_{23} = \{ 123, 126, 157, 234, 345, 346, 456, 457 \}.$

$E_{24} = \{ 12, 23, 34, 36, 45,  67 \}$,
$\cE_{24} = \{ 123, 127, 145, 234, 236, 345, 367, 567 \}.$

\medskip
\noindent
$n=8$ vertices:

\smallskip
$E_{27} = \{ 17, 18, 27, 37, 47, 57, 67 \}$,
$\cE_{27} = \{ 127, 138, 167, 178, 237, 347, 457, 567 \}.$

$E_{28} = \{ 15, 25, 35, 45, 56, 67, 68 \}$,
$\cE_{28} = \{ 125, 145, 235, 345, 567, 568, 678 \}.$

$E_{29} = \{ 14, 24, 34, 45, 56, 57, 58 \}$,
$\cE_{29} = \{ 124, 145, 234, 345, 567, 568, 578 \}.$

$E_{30} = \{ 16, 17, 26, 36, 46, 56, 78 \}$,
$\cE_{30} = \{ 126, 156, 167, 178, 236, 278, 346, 456 \}.$

$E_{31} = \{ 16, 17, 26, 28, 36, 46, 56 \}$,
$\cE_{31} = \{ 126, 137, 156, 167, 236, 248, 268, 346, 456 \}.$

$E_{32} = \{ 15, 16, 25, 35, 45, 67, 68 \}$,
$\cE_{32} = \{ 125, 145, 167, 168, 235, 345, 678 \}.$

$E_{33} = \{ 15, 16, 17, 25, 35, 45, 78 \}$,
$\cE_{33} = \{ 156, 157, 167, 178, 235, 245, 278, 345 \}.$

$E_{34} = \{ 12, 18, 23, 24, 25, 56, 57 \}$,
$\cE_{34} = \{ 123, 124, 168, 178, 234, 256, 257, 567 \}.$

$E_{35} = \{ 12, 23, 24, 45, 48, 56, 57 \}$,
$\cE_{35} = \{ 123, 124, 234, 248, 456, 457, 458, 567 \}.$

$E_{36} = \{ 12, 23, 24, 28, 45, 56, 67 \}$,
$\cE_{36} = \{ 123, 124, 128, 167, 234, 238, 245, 456, 567  \}.$

$E_{37} = \{ 12, 23, 24, 25, 38, 56, 67 \}$,
$\cE_{37} = \{ 123, 124, 167, 234, 238, 245, 256, 378, 567  \}.$

$E_{38} = \{ 12, 23, 34, 36, 38, 45, 67 \}$,
$\cE_{38} = \{ 123, 127, 145, 234, 236, 345, 348, 367, 368, 567  \}.$

$E_{39} = \{ 12, 23, 24, 45, 48, 56,  67 \}$,
$\cE_{39} = \{ 123, 124, 167, 234, 245, 248,  456, 458, 567  \}.$

$E_{40} = \{ 12, 23, 24, 45,  56,  67, 68 \}$,
$\cE_{40} = \{ 123, 124, 167, 234, 245,  456,  567, 568, 678  \}.$

$E_{41} = \{ 12, 23, 24, 45, 56, 57, 68 \}$,
$\cE_{41} = \{ 123, 124, 168, 234, 456, 457, 567, 568  \}.$

$E_{42} = \{ 12, 23, 34, 38, 45,  46, 57 \}$,
$\cE_{42} = \{ 123, 126, 157, 234, 238, 345, 346, 348, 456, 457 \}.$

$E_{43} = \{ 12, 23, 28, 34, 36, 45,  67 \}$,
$\cE_{43} = \{ 123, 127, 128, 145, 234, 236,238, 345, 367, 567 \}.$

$E_{44} = \{ 12, 23, 24, 45,  56, 67 , 78 \}$,
$\cE_{44} = \{ 123, 124, 178, 234, 245, 456, 567, 678 \}.$

$E_{45} = \{ 12, 23, 34, 36, 45,  67 , 78 \}$,
$\cE_{45} = \{ 123, 127, 145, 234, 236, 345, 367, 578, 678 \}.$

$E_{46} = \{ 12, 23, 24, 38, 45,  56, 67 \}$,
$\cE_{46} = \{ 123, 124, 167, 234, 238, 245, 378,  456, 567 \}.$

$E_{47} = \{ 12, 23, 34, 45,  46, 57, 78 \}$,
$\cE_{47} = \{ 123, 126, 178, 234, 345, 346, 456, 457, 578 \}. \hqed$

\end{proof}

\subsection{Inductive step}
\label{Sec_step}

Let $G=(V, E)$ be a graph and $s \ge 1$.

\begin{definition}
A path $l=(v_0, e_1, v_1, \ldots, v_{s-1}, e_s, v_s)$ is referred to be a {\em leg} (of {\em length} $s$) in $G$ if and only if
\renewcommand{\labelenumi}{(\roman{enumi})}
\begin{enumerate}
\item \quad $V(l) = \{ v_0, \ldots, v_s \} \subseteq V$;
\item \quad $E(l) = \{ e_1, \ldots, e_s \} \subseteq E$;
\item \quad $d_G(v_0) \ge 3, \, d_G(v_1) = \ldots = d_G(v_{s-1}) = 2, \, d_G(v_s)=1$.
\end{enumerate}
\end{definition}

The vertex $v_0$ is the {\em joint} or {\em joint vertex} and $v_s$ is the {\em end vertex} of $l$.
Clearly, every graph $G$ with minimum degree $\delta(G) = 1$ and maximum degree $\Delta(G) \ge 3$ has a leg.
Moreover, each tree $T$ being not a path $P_n$ $(n \ge 1)$ has at least three legs.

\begin{definition}
The graph $G \ominus l = ( V', E')$ results from $G=(V, E)$ by {deleting the leg}  \linebreak $l=(v_0, e_1, v_1, \ldots, v_{s-1}, e_s, v_s)$  if and only if \/
$V' = V \setminus \{ v_1, \ldots, v_s \}$ and $E' = E \setminus \{e_1, \ldots, e_s \}$.
\end{definition}

Note that $G \ominus l$ is connected if and only if $G$ is connected, since the joint vertex $v_0$ is not deleted by  the deletion of the leg $l$ in $G$.

\bigskip
\noindent
{\bf Proof of Theorem 6.}

\smallskip
For the induction basis see subsection 3.1. Note that all trees $T=(V, E)$ with 7 or 8 vertices are edge intersection hypergraphs of a 3-uniform hypergraph (cf. Lemma 3).

\smallskip
{\em Induction hypothesis:} Every  tree $T=(V, E)$ with $7 \le |V| \le n$ is an edge intersection hypergraph of a 3-uniform hypergraph.

\smallskip
Let $n \ge 8$, $T'=(V', E')$ be a tree with $|V'|=n+1$ vertices; because of Lemma 1 we can exclude stars and paths from our considerations.
Therefore $T'$ has at least three end vertices and also at least three legs. Let $v_0$ and $v_s$ be the joint vertex and the end vertex of a shortest leg $ l=( v_0, v_1, \ldots, v_s )$, respectively.

We  delete the leg $l$ in $T'$ and obtain $T =(V, E) = T' \ominus l$. Obviously, $v_0 \in V$ and $v_1, \ldots, v_s \in V' \setminus V$. According to the length $s$ of the leg $l$ we consider two cases.

\medskip
\noindent
\ul{Case 1: $s = 1$.}

\smallskip
Because of $d_T(v_0) \ge 2$ there are at least two neighbors $u \neq u'$ of $v_0$ in the tree $T$. Moreover, we have $|V| = n \ge 8$ and  the induction basis implies the existence of a hypergraph $\cH = (V, \cE)$ with $T = EI(\cH)$ and $\cH$ is 3-uniform.

Consider $\cE' := \cE \, \cup \, \{ \{ u, v_0, v_1 \}, \{ u', v_0, v_1 \} \}$ and the 3-uniform hypergraph $\cH' = ( V', \cE')$. Then $ \{v_0, v_1 \} =  \{u,  v_0, v_1 \} \, \cap \, \{ u', v_0, v_1 \} $.

Clearly, $ \{ u, v_0, v_1 \} \, \cap \, V =  \{ u, v_0 \}$ and $ \{ u', v_0, v_1 \} \, \cap \, V =  \{ u', v_0 \}$. Taking an arbitrarily chosen hyperedge $e \in \cE' \,  \setminus \, \{ \{ u, v_0, v_1 \}, \{ u', v_0, v_1 \} \} = \cE $, the only edge which can result from  the intersection $ \{ u, v_0, v_1 \} \, \cap \, e $ and $ \{ u', v_0, v_1 \} \, \cap \, e $ in $EI(\cH')$ is the edge $\{ u, v_0 \} \in E(T)$ and $ \{ u', v_0 \} \in E(T)$, respectively. Consequently, the hypergraph $\cH'$
 has the edge intersection hypergraph $EI(\cH') = T'$.

\medskip
\noindent
\ul{Case 2: $s \ge 2$.}

\smallskip
Let $l, l', l''$ be three legs  in $T'=(V', E')$. Deleting the legs $l, l', l''$ in $T'$, we would obtain a new tree  with at least one vertex.
Since $l = ( v_0, v_1, \ldots, v_s )$  is a shortest leg and $|V'| = n+1 $ is valid, the leg $l$ contains at most $\frac{n}{3}$ vertices.
%
Because of $n \ge 8$, the deletion of the  leg $l$ corresponds to the deletion of the vertices $v_1, \ldots, v_s$ in $T'$ and leads to the tree $T = (V, E)$ with
$ | V | = n + 1 - s \ge n + 1 - \frac{n}{3} = \frac{2}{3} n + 1 \ge
\frac{19}{3} > 6$. Therefore, $T$ has at least 7 vertices; we apply the induction basis and obtain the existence of a 3-uniform hypergraph $\cH = (V, \cE)$ with $T = EI(\cH)$.

Now we construct the hypergraph $\cH' = (V', \cE')$ from $\cH = (V, \cE)$.

In comparison to $T$, in $T'$ we find  the additional edges $\{ v_0, v_1 \}, \{v_1, v_2 \}, \ldots, \{v_{s-1}, v_s \}$ which have to be generated by certain hyperedges of $\cH'$.
We add the following three types of hyperedges to the hypergraph $\cH$.

\begin{itemize}
\item The first one is the hyperedge $\{u, v_0, v_1 \}$, where $u \in V$ is a neighbor of the vertex $v_0$ in the tree $T$. Because of $v_1 \notin V(T)$, the only edge being induced by this hyperedge and the hyperedges of $\cE(\cH)$ in the edge intersection hypergraph of $\cH_0 = ( V \cup \{ v_1 \}, \cE \cup \{ \{ u, v_0, v_1 \} \})$ is the edge $\{ u, v_0  \} \in E(T)$.
\item The second set of new hyperedges consists of $\{v_0, v_1, v_2 \}, \{v_1, v_2, v_3 \}, \ldots, \{v_{s-2}, v_{s-1}, v_s \}$.
    Adding these hyperedges (and the vertices $v_2, \ldots, v_s$) to $\cH_0$ we obtain a hypergraph  $\cH_1$; in the corresponding edge intersection hypergraph $EI(\cH_1)$ we find  the new edges\\
    $\{ v_0, v_1 \} = \{u, v_0, v_1 \} \, \cap \, \{v_0, v_1, v_2 \},
    \{ v_1, v_2 \} = \{v_0, v_1, v_2 \}  \, \cap \, \{v_1, v_2, v_3 \},  \ldots,
     \{ v_{s-2}, v_{s-1} \} =$ \\ $ \{v_{s-3}, v_{s-2}, v_{s-1} \}  \, \cap \, \{v_{s-2}, v_{s-1}, v_{s} \}$ and not more (because of $v_1, \ldots, v_s \notin V(T)$).
\item To obtain the last edge needed in $T' = EI(\cH')$, we choose a vertex $w \in V(T) \setminus  \{ v_0 \}$ being not a neighbor of $v_0$. The existence of such a vertex $w$ becomes clear since all legs  in $T'$ have to have a length of at least 2. Therefore, for $w$ we can choose an  end vertex of an arbitrary leg $l' \neq l$ in $T$.  Considering $\{ w, v_{s-1}, v_{s} \}$, we see that $\{ w, v_{s-1}, v_{s} \} \, \cap \,  \{v_{s-2}, v_{s-1}, v_s \} =  \{ v_{s-1}, v_s \} $.

    We add the new hyperedge $\{ w, v_{s-1}, v_{s} \}$ to the hypergraph $\cH_1$ and obtain the hypergraph $\cH'$. For two reasons,  $\{ v_{s-1}, v_{s} \}$ is the only edge being generated by the hyperedge $\{ w, v_{s-1}, v_{s} \}$ in $EI(\cH')$.


    \begin{enumerate}
    \item[(i)] $|V(T) \, \cap \, \{ w, v_{s-1}, v_{s} \}| = 1$, therefore the intersection of $\{ w, v_{s-1}, v_{s} \}$ with any hyperedge of the original hypergraph $\cH = ( V, \cE) = (V(T), \cE)$ cannot lead to an additional edge in $EI(\cH')$.
    \item[(ii)] Because of $w \in V \setminus \{u \}$, the intersection of $\{ w, v_{s-1}, v_{s} \}$ with one of the "new" hyperedges $\{u, v_0, v_1 \}, \{v_0, v_1, v_2 \}, \{v_1, v_2, v_3 \}, \ldots, \{v_{s-2}, v_{s-1}, v_s \}$ is always a subset of $\{ v_{s-1}, v_s \} $. \\Hence the edge $\{ v_{s-1}, v_s \} \in E(T')$ is the only edge being induced by $\{ w, v_{s-1}, v_{s} \}$ in $EI(\cH')$. \hqed
    \end{enumerate}

\end{itemize}

\section{Concluding remarks}

       In Corollary 1 and in Section 3 we characterized those cycles and trees which are edge intersection hypergraphs of 3-uniform hypergraphs. In connection with these results several interesting problems occur.

\medskip
{\bf Problem 2.}
 Characterize other classes ${\cal G}$  of graphs being edge intersection hypergraphs of 3-uniform hypergraphs.

\medskip
Obviously, an analog question can be asked for $k$-uniform ($k \ge 3$) instead of 3-uniform hypergraphs $\cH$, maybe in combination with a minimum number of hyperedges in $\cH$.

\medskip
{\bf Problem 3.}
Let ${\cal G}$ be a class of graphs, $k \ge 3$, $n_0 \in \N^+ $, $n \ge n_0$ and $G_n \in {\cal G}$ a graph with $n$ vertices. What is the minimum cardinality $| \cE |$ of the edge set of a $k$-uniform hypergraph $\cH_n =(V, \cE)$ with $EI( \cH_n) = G_n$?

\medskip

We conjecture that the solution of Problem 3 becomes more difficult for $k > 3$.

\end{document}